\documentclass[12pt,reqno,tbtags]{amsart}
\usepackage{amsmath,amsfonts,amssymb,enumerate}

\newtheorem{theorem}{Theorem}[section] 

\newtheorem*{theorem*}{Theorem} \newtheorem{lemma}[theorem]{Lemma}
\newtheorem{proposition}[theorem]{Proposition}

\theoremstyle{definition}{

   }

\theoremstyle{remark}{

  \newtheorem{remark}{Remark}[section] \newtheorem*{remark*}{Remark}
   }

\begin{document}

\newcommand{\rr}{\mbox{\boldmath $r$}}
\def\mod#1{\ifmmode\;\,\hbox{\rm (mod $#1$)}\else$(\mathop{\rm mod}{#1})$\fi}

\title{On the Asymptotic Existence of Hadamard Matrices}

 \author{Warwick de Launey}
 \address{Center for Communications Research \\
        Institute For Defense Analyses \\
	4320 Westerra Court \\
	San Diego, California, CA92121}
 \email{warwick@ccrwest.org}
 \email{warwickdelauney@earthlink.net}
 \thanks{Work by a Contractor to the US Government}

\begin{abstract}{It is conjectured that Hadamard matrices exist for all
   orders $4t$ ($t>0$).  However, despite a sustained effort over more
   than five decades, the strongest overall existence results are
   asymptotic results of the form: for all odd natural numbers $k$,
   there is a Hadamard matrix of order $k2^{[a+b\log_2k]}$, where $a$
   and $b$ are fixed non-negative constants.  To prove the Hadamard
   Conjecture, it is sufficient to show that we may take $a=2$ and
   $b=0$.  Since Seberry's ground-breaking result, which showed that
   we may take $a=0$ and $b=2$, there have been several improvements
   where $b$ has been by stages reduced to $3/8$.  In this paper, we
   show that for all $\epsilon>0$, the set of odd numbers $k$ for
   which there is a Hadamard matrix of order $k2^{2+[\epsilon\log_2k]}$
   has positive density in the set of natural numbers.  The proof
   adapts a number-theoretic argument of Erdos and Odlyzko to show
   that there are enough Paley Hadamard matrices to give the result.}
\end{abstract}

 \keywords{Hadamard matrices, asymptotic existence, cocyclic Hadamard
   matrices, relative difference sets, Riesel numbers, Extended
   Riemann Hypothesis.}

\maketitle

\section{Overview}
As noted in the abstract, there have been incremental improvements in
the power of known asymptotic existence results for Hadamard matrices
\cite{Crai95, CrKhHo97, Sebe76}.  These theorems all have the form:
For all positive odd integers $k$, there is a Hadamard matrix of order
$k2^{[a+b\log_2k]}$, where $a$ and $b$ are fixed non-negative real numbers.
The strength of the result depends on how close $b$ is to zero, and
then on how close $a$ is to zero.  In this paper, we adapt a
number-theoretic argument of Erdos and Odlyzko \cite{ErOd79} to prove
the following theorem.
\begin{theorem}\label{MainThm}
  Let $\epsilon>0$.  Let $H(x)$ denote the number of odd positive
  integers $k\leq x$ for which there is a Hadamard matrix of order
  $2^{\ell}k$, for some positive integer $\ell\leq2+\epsilon\log_2 k$.
  Then there is a constant $c_1(\epsilon)$, dependent only on
  $\epsilon$, such that, for all sufficiently large $x$,
  $H(x)>c_1(\epsilon)x$.
\end{theorem}

Our approach is to prove the following number-theoretic result.
\begin{theorem}\label{MainNumberThm}
  Let $\epsilon>0$.  Let $M_\epsilon(x)$ denote the number of odd
  positive integers $k\leq x$ for which $2^mk-1$ is prime for some
  positive integer $m\leq\epsilon\log_2k$.  Then there is a constant
  $c_2(\epsilon)$, dependent only on $\epsilon$, such that for all
  sufficiently large $x$, $M_\epsilon(x)>c_2(\epsilon)x$.
\end{theorem}
Since there is a Paley Hadamard matrix of order $q+1$, when
$q\equiv3\mod{4}$ is prime, and of order $2(p+1)$, when
$p\equiv1\mod{4}$ is prime, taking a Kronecker product with the
Sylvester Hadamard matrix of appropriate order, we have, for any odd
prime $p$, a Hadamard matrix of order $2^m(p+1)$, where
\[
m\geq
\left\{
\begin{array}{ll}
1 & \mbox{if $p\equiv1\mod{4}$,}\\
0 & \mbox{if $p\equiv3\mod{4}$.}  
\end{array}
\right.
\]
So, for $k$ odd, when $2^mk-1$ equals a prime $p$, for some positive
integer $m<\epsilon\log_2k$, there is a Hadamard matrix of order
$2^mk$ if $m>1$, and of order $2^{2}k$ if $m=1$.  So, in either case,
there is a Hadamard matrix of order $2^\ell k-1$, for some positive integer
$\ell\leq2+\epsilon\log_2k$.  Therefore, Theorem~\ref{MainNumberThm}
certainly implies Theorem~\ref{MainThm} with
$c_1(\epsilon)=c_2(\epsilon)$.
\begin{remark} 
  1.  Since the Hadamard matrices used above are all cocyclic (see
  \cite{DeFlHo00} for this fact and a discussion of cocyclic Hadamard
  matrices), Theorem~\ref{MainThm} also holds for cocyclic Hadamard
  matrices.  Thus we have an asymptotic existence result for a certain
  class of relative difference sets.\newline 2.  Since there is a
  Sylvester matrix of order $4$, there is a constant $c_1'(\epsilon)$
  dependent only on $\epsilon$, such that $H(x)\geq c_1'(\epsilon)\,x$
  for all $x\geq1$.
\end{remark}

There are inherent limitations to our approach.  The Hadamard Conjecture
implies that we may take $c_1(\epsilon)=1/2$.  However, the following holds.
\begin{proposition}\label{Proposition}
  There is a positive number $\delta_0$, such that, for all
  $\epsilon>0$, there are infinitely many $x\geq0$ for which
  $M_\epsilon(x) \leq \frac{1}{2}(1-\delta_0)x$.  Moreover, if the
  Extended Riemann Hypothesis holds, then for all $\epsilon,\delta>0$,
  we have $M_\epsilon(x)\leq x2(1+\delta)\log_2(1+\epsilon)$ for all
  sufficiently large $x$.
\end{proposition}
The part concerning the Extended Riemann Hypothesis will be proved in
Section~\ref{PropositionSection}.  We explain now why the first part
of the proposition holds.  H.  Riesel \cite{Ries56} showed that there
are infinitely many odd numbers $k$ for which $2^mk-1$ is always
composite.  The smallest known such number is $509203$.  Indeed, it
can be shown that for any $m\geq0$, at least one of the primes
$3,5,7,13,17$ or $241$ divides $2^mk-1$, where $k=509203+11184810r$,
and $r$ is any non-negative integer.  Consequently, for any
$\epsilon>0$, there are infinitely many $x\geq0$ such that
$M_\epsilon(x)\leq x(1 -1/11184810)/2$.  This proves the first part of
the proposition, and shows that the approach to $c_1(\epsilon)$ via
$c_2(\epsilon)$ will fail once $c_1(\epsilon)$ becomes close enough to
$1/2$.  Indeed, the proposition says that if the Extended Riemann
Hypothesis holds, then the approach will fail once we ask that
$c_1(\epsilon)$ exceed $2\log_2(1+\epsilon)$, a number which, when
$\epsilon$ is small, is far less than 1/2.

On the other hand, by using more involved number-theoretic arguments
to explore the scope of known constructions for Hadamard matrices, one
might be able to show that $c_1(\epsilon)$ can be taken very close to
1/2.  Aside from offering us a way to prove strong asymptotic existence
results for Hadamard matrices, this approach has the advantage of
giving us a measure of how far we have come towards proving the
Hadamard Conjecture.  We now describe a replacement constant
$c_3(\epsilon)$ for $c_2(\epsilon)$.

Notice that if there is a Hadamard matrix of order
$k_i2^{[\epsilon\log_2k_i]}$ for $i=1,2,\dots,n$, then there is a
Hadamard matrix of order $k2^{[\epsilon\log_2k]}$ for $k=k_1k_2\dots
k_n$.  So define $M'_\epsilon(x)$ to be the number of odd positive
integers $k$ with the property P, say, that $k=k_1k_2\dots k_n$ where,
for $i=1,2,\dots,n$, there are $m_i\leq\epsilon\log_2k_i$ such that
$k_i2^{m_i}-1$ is prime for $i=1,2,\dots,n$.  Clearly,
Theorem~\ref{MainNumberThm} implies that there is a constant
$c_3(\epsilon)>0$ such that $M'_\epsilon(x)\geq c_3(\epsilon)x$ for
all sufficiently large $x$.  Notice that, by the Prime Number Theorem
for primes in arithmetic progression, there are infinitely many Riesel
numbers which are prime.  Therefore there are infinitely many numbers
which do not have property P.  However, notice that if $k_1$ and $k_2$
have property P, then so does $k=k_1k_2$, and, since most large
numbers $k$ can be written in the form $k_1k_2$ in many ways, it seems
likely that almost all numbers have property P.  So more complicated
counting arguments along the lines of those described in this paper
might be used to prove that $c_1(\epsilon)$ can be taken very close to
1/2.

The rest of this paper is organized as follows.  In
Section~\ref{TheoremSection}, we prove Theorem~\ref{MainNumberThm}.
Then in Section~\ref{PropositionSection}, assuming a lemma concerning
the Extended Riemann Hypothesis, we prove
Proposition~\ref{Proposition}.  Finally, in
Section~\ref{LemmaSection}, we prove the lemma needed in
Section~\ref{PropositionSection}.

 \section{Proof of  Theorem~1.2}\label{TheoremSection}
We adapt an argument of Erdos and Odlyzko \cite{ErOd79} to prove the
following lemma.  The lemma and our comments about how large one can
take $c_4(\epsilon)$ is of independent number-theoretic interest.
\begin{lemma}\label{NumberLemma}
  Let $\epsilon>0$.  Let $N_\epsilon(x)$ denote the number of positive
  integers $k\leq x$ for which $2^mk-1$ is prime for some positive
  integer $m<\epsilon\log_2x$.  Then there is a constant
  $c_4(\epsilon)$, dependent only on $\epsilon$, such that, for all
  sufficiently large $x$, $N_\epsilon(x)>c_4(\epsilon)x$.
\end{lemma}
Before we prove this lemma, we confirm that it implies
Theorem~\ref{MainNumberThm}.  In fact, we prove the following stronger
result.
\begin{lemma}
  Lemma~\ref{NumberLemma} and Theorem~\ref{MainNumberThm} are
  equivalent.  Under this equivalence, the constants $c_4(\epsilon)$
  and $c_2(\epsilon)$ are related as follows:
\begin{equation}\label{c2c4Inequality}
c_4(\epsilon) = c_2(\epsilon)>(1-\delta)c_4(\epsilon/A)\,,
\end{equation}
for all $A>1$ and $\delta>0$.
\end{lemma}
\begin{proof}
First notice that Theorem~\ref{MainNumberThm} implies
Lemma~\ref{NumberLemma} with $c_4(\epsilon)= c_2(\epsilon)$.  So it
is sufficient to prove that Lemma~\ref{NumberLemma} implies
Theorem~\ref{MainNumberThm} with
$c_2(\epsilon)>(1-\delta)c_4(\epsilon/A)$ for all $A>1$ and
$\delta>0$.

Fix $x$.  Then, for all $A>B>1$,
\[
M_{A\epsilon}(x)
>N_{\epsilon}(x)-N_{A\epsilon}(x^{1/A})
>c_4(\epsilon)x-x^{1/B}=c_4(\epsilon)(1-x^{-1+1/B})x\,.
\]
So, for all $A>B>1$ and $\delta>0$,
\[
M_{\epsilon}(x)>x(1-\delta)c_4(\epsilon/A)\,,
\] 
for all $x>\delta^{B/(1-B)}$.  Consequently, we may take
$c_2(\epsilon)>(1-\delta)c_4(\epsilon/A)$, for all $A>1$ and
$\delta>0$.  This completes the proof of the lemma.
\end{proof}
Therefore, since Theorem~\ref{MainNumberThm} implies
Theorem~\ref{MainThm} with $c_1(\epsilon)=c_2(\epsilon)$,
Lemma~\ref{NumberLemma} implies Theorem~\ref{MainThm} with 
\[
c_1(\epsilon)>(1-\delta)c_4(\epsilon/A)\,,
\]
for all $A>1$ and $\delta>0$.

We now prove Lemma~\ref{NumberLemma}.  Fix $\epsilon>0$, an integer
$x\geq4^{1/\epsilon}$, and set $L=[\epsilon\log_2x]-1$.  So $L\geq1$,
and $L+1$ is the largest integer less than or equal to
$\epsilon\log_2x$.  In particular,
\begin{equation}\label{LConstraint}
x\geq (2^{L+1})^{1/\epsilon}\,,
\end{equation}
and
\begin{equation}\label{LEquality}
L=\epsilon\log_2 x -\alpha\qquad \mbox{where $2>\alpha\geq1$.}
\end{equation}
For odd $k\leq x$, let $S(k,x)$ denote the number of primes of the
form $2^\ell k-1$ where $\ell=1,2,\dots,L$.  Then 
\[
N_\epsilon(x)\geq\sum_{k\leq x\atop{odd}}\mathbf{1}\{S(k,x)>0\}\,.
\]
We show that the variance of $S(k,x)$ is quite small.  We have the
following analog to a special case of Lemma 2 of \cite{ErOd79}.
\begin{lemma}\label{ErOdLemma1}
  There exists a constant $c_5(\epsilon)$, dependent on $\epsilon$
  only, such that, for all sufficiently large $x$,
\[
\sum_{{k\leq x}\atop{odd}}S^2(k,x)\leq c_5(\epsilon) x\,.
\]
\end{lemma}
\begin{proof}
  In their paper \cite{ErOd79}, Erdos and Odlyzko fix $r$ base
  primes\linebreak[4] $p_1,p_2,\dots,p_r$, and, for each positive
  integer $k\leq x$ coprime to $p_1p_2\dots p_r$, define the quantity
  $R(k,x)$ to be the number of $r$-tuples $(a_1,a_2,\dots, a_r)$
  ($a_i=1,2,\dots,N$) such that $1+k
  \prod_{a_1,a_2,\dots,a_r=1}^Np_i^{a_i}$ is prime.  Here $N\sim
  c'\log x$, where $c'$ is a fixed constant.  They assert that there
  is a constant $c''$, which depends only on the choice of base primes
  and $c'$, such that
\begin{equation}\label{ErOdREqn}
\sum_{{k\leq x}\atop{(k,\,p_1p_2\dots p_r)=1}}R^2(k,x)\leq c'' x(\log x)^{2r-2}\,.
\end{equation}
Our quantity $S(k,x)$ is analogous to their quantity $R(k,x)$ with
$r=1$, $p_1=2$, $N=L$, and $c'=\epsilon/\log 2$.  With these settings,
their quantity $R(k,x)$ is defined for odd $k\leq x$ and counts the
number of integers $\ell\in\{1,2,\dots,L\}$ for which $2^\ell k+1$ is
a prime.  On the other hand, our quantity $S(k,x)$ is defined for odd
$k\leq x$ and counts the number of integers $\ell\in\{1,2,\dots,L\}$
such that $2^\ell k-1$ is prime.  As pointed out at the end of the
introduction to their paper, their techniques handle primes of the
form $2^\ell k-1$ in the same way as they set out for primes of the
form $2^\ell k+1$.  In particular, equation (\ref{ErOdREqn}) also
holds for $r=1$, and $p_1=2$, when our quantity $S(k,x)$ replaces the
analogous quantity $R(k,x)$.
\end{proof}

Following Erdos and Odlyzko, we have by the Cauchy-Schwarz
inequality
\[
N_\epsilon(x)\geq \Bigl(\sum_{k\leq x\atop{odd}}S(k,x)\Bigr)^2
/\sum_{k\leq x\atop{odd}}S^2(k,x)
\geq\frac{1}{c_5(\epsilon)x}\Bigl(\sum_{k\leq x\atop{odd}}S(k,x)\Bigr)^2\,.
\]
To prove the lemma, it is therefore sufficient to show there is a
constant $c_6(\epsilon)$ dependent only on $\epsilon$, such that for
all sufficiently large $x$
\begin{equation}\label{SSumEqn}
\sum_{k\leq x\atop{odd}}S(k,x)\geq c_6(\epsilon)x\,. 
\end{equation}
For then we may take 
\begin{equation}\label{c4Equality}
c_4(\epsilon)=c_6(\epsilon)^2/c_5(\epsilon)\,.
\end{equation}
Let $\pi(x;q,a)$ denote the number of primes $p\leq x$ such that
$p\equiv a\mod{q}$.  Then, since $\pi(2^\ell
x-1;2^{\ell+1},2^{\ell}-1)=\pi(2^\ell x;2^{\ell+1},2^{\ell}-1)$, and,
since $p=2^\ell k-1$ is prime for some odd positive integer $k\leq x$
if and only if $p\leq 2^\ell x$ is a prime congruent to $2^\ell-1$
modulo $2^{\ell+1}$, we have
\[
\sigma_\epsilon(x)=\sum_{k\leq x\atop{odd}}S(k,x)
=\sum_{\ell=1}^L\pi(2^{\ell}x;2^{\ell+1},2^\ell-1)\,.
\]
To estimate the sum on the right, we use the following lemma which is
a special case of Lemma 1 of Erdos and Odlyzko \cite{ErOd79}.
\begin{lemma}\label{ErOdLemma2}
  There exist constants $c_7$ and $c_8$ such that, for all integers
  $\ell\geq1$,
\[
\pi(x;2^{\ell+1},2^\ell-1)\geq \frac{c_7x}{2^\ell\log x}\,,
\qquad\mbox{for all $x\geq (2^{\ell+1})^{c_8}$.}
\]
\end{lemma}
Now by the inequality (\ref{LConstraint}), for $\epsilon\leq1/c_8$, we
have
\[
x\geq(2^{L+1})^{1/\epsilon}\geq(2^{L+1})^{c_8}\geq(2^{\ell+1})^{c_8}\,,
\]
for all $\ell=1,2,\dots,L$.  Therefore, for $\epsilon\leq1/c_8$, we
may use Lemma~\ref{ErOdLemma2} to bound $\sigma_\epsilon(x)$ below.
We obtain
\begin{equation}\label{c7Sum}
\sigma_\epsilon(x)\geq \sum_{\ell=1}^L\frac{c_7x}{\log 2^\ell x}
=c_7\frac{x}{\log 2}\sum_{\ell=1}^L\frac{a}{1+\ell a}\,,
\end{equation}
where $a=(\log_2 x)^{-1}$.  Now, for $L\geq M\geq1$, define
\[
I(M,L,a)=\int_M^L\frac{a}{1+\ell a}d\ell\,.
\]
Then
\begin{equation}\label{IMLaEquality}
I(M,L,a)=\log\left(\frac{1+La}{1+Ma}\right)\,,
\end{equation}
and, since, for all $a>0$, the function $f_a(\ell)=a/(1+\ell a)$ is
monotonic decreasing in the region $\ell\geq0$, we have, for all
$L\geq M\geq1$,
\begin{equation}\label{IMLaInequality}
  I(M-1,L-1,a)\quad> \quad\sum_{\ell=M}^L\frac{a}{1+\ell a}\quad> \quad
  I(M,L,a)\,.
\end{equation}
So,
\[
\sigma_\epsilon(x)>
x\frac{c_7}{\log2}\log\left(\frac{1+L(\log_2x)^{-1}}{1+(\log_2x)^{-1}}\right)\,.
\]
Now, by equation (\ref{LEquality}),
$L(\log_2x)^{-1}=\epsilon-\alpha(\log_2x)^{-1}$, where $2\geq\alpha>1$.  So,
for some $\beta\in[2,3)$,
\[
\sigma_\epsilon(x)>
x\frac{c_7}{\log2}
\log\left(1+\frac{\epsilon-\beta(\log_2x)^{-1}}{1+(\log_2x)^{-1}}\right)\,,
\]
and, for all $\delta>0$, 
\[
\sigma_\epsilon(x)>x(1-\delta)c_7\log_2(1+\epsilon)
\]
for all sufficiently large $x$.  Thus we have proved the inequality
(\ref{SSumEqn}) for all $c_6(\epsilon)<c_7\log_2(1+\epsilon)$.  Therefore,
by equation (\ref{c4Equality}), Lemma~\ref{NumberLemma} holds for all
$c_4(\epsilon)<(c_7\log_2(1+\epsilon))^2/c_5(\epsilon)$.
\section{Proof of the Proposition}\label{PropositionSection}
It will be sufficient to prove that, for all sufficiently large $x$,
for all $\delta>0$,
\begin{equation}\label{NepsEqn}
N_\epsilon(x)\leq 2x(1+\delta)\log_2(1+\epsilon)\,.
\end{equation}
Notice that, since, by definition,
$N_\epsilon(x)\leq\frac{1}{2}(1+\frac{1}{x})x$, the above inequality
(\ref{NepsEqn}) holds trivially when $\epsilon > 2^{\frac{1}{4}}-1$.
So we may certainly suppose that $\epsilon<1$.

Now, since $S(k,x)>0$ implies $S(k,x)\geq1$, we have
\[
N_\epsilon(x)<\sum_{k\leq x\atop{odd}}S(k,x)=\sigma_\epsilon(x)\,.
\]
We will use the following technical lemma to be proved in the next
section.
\begin{lemma}\label{ERHLemma}
  Suppose the Extended Riemann Hypothesis holds.  Then there is a
  constant $A>0$ such that, for all positive coprime non-negative
  integers $q$ and $a<q$,
\[
\pi(x;q,a)<\frac{2x}{\phi(q)\log x}+Ax^{1/2}\log x\,.
\]
\end{lemma}
Assuming this lemma, we have
\begin{eqnarray*}
\sigma_\epsilon(x)&=&\sum_{\ell=1}^L\pi(2^{\ell}x;2^{\ell+1},2^\ell-1)\\
&<&LA(2^{L}x)^{1/2}\log(2^Lx)+\sum_{\ell=1}^L\frac{2x}{\log 2^\ell x}\,.
\end{eqnarray*}
By the inequalities (\ref{LConstraint}) and (\ref{LEquality}), $2^L
x<x^{1+\epsilon}$ and $L<\epsilon\log_2x$.  So the first term is less
than $\epsilon(1+\epsilon) Ax^{(1+\epsilon)/2}(\log_2 x)^2$, which is
negligible since we have supposed that $\epsilon<1$.  Moreover, the
second term is the sum in equation (\ref{c7Sum}) with $c_7=2$.  This
sum is handled as before, except we use the upper bound in the
inequalities (\ref{IMLaInequality}).  Therefore, for $\epsilon<1$, any
$\delta>0$, and all sufficiently large $x$,
\begin{eqnarray*}
N_\epsilon(x)&<&
\sigma_\epsilon(x)\\
&<&(1+\delta)\frac{2x}{\log2}I(0,L-1,(\log_2x)^{-1})\\
&=&x2(1+\delta)\log_2(1+\epsilon-\beta(\log_2x)^{-1})\\
&\leq& x2(1+\delta)\log_2(1+\epsilon)\,.
\end{eqnarray*}
\section{Proof of Lemma 3.1}\label{LemmaSection}
To prove the lemma, we consider the familiar number-theoretic function
\[
\psi(x;q,a)=\sum_{k\leq x\atop{k\equiv a\,(\mathrm{mod}\,q)}}\Lambda(k)\,,
\]
where $\Lambda(k)$ is the von Mangoldt function defined as follows:
\[
\Lambda(k)=\left\{
\begin{array}{ll}
\log p&\mbox{if $k$ is a power of the prime $p$},\\
0&\mbox{if $k$ is not a prime power}.
\end{array}
\right.
\]
Now
\[
\pi(x;q,a)=\sum_{k\leq x\atop{k\equiv a\,(\mathrm{mod}\,q)}}
\mathbf{1}\{\mbox{$k$ is prime}\}\,,
\]
and
\begin{eqnarray*}
\psi(x;q,a)
&=&\sum_{k\leq x\atop{k\equiv a\,(\mathrm{mod}\,q)}}
\mathbf{1}\{\mbox{$k$ is prime}\}\lfloor\log_k x\rfloor\log k\\
&=&\sum_{\sqrt{x}<k\leq x\atop{k\equiv a\,(\mathrm{mod}\,q)}}
\mathbf{1}\{\mbox{$k$ is prime}\}\log k\\
&&\qquad+\sum_{k\leq \sqrt{x}\atop{k\equiv a\,(\mathrm{mod}\,q)}}
\mathbf{1}\{\mbox{$k$ is prime}\}\lfloor\log_k x\rfloor\log k\,.
\end{eqnarray*}
Since $\lfloor\log_k x\rfloor\log k\leq\log x$, we therefore have
\[
\sum_{\sqrt{x}<k\leq x\atop{k\equiv a\,(\mathrm{mod}\,q)}}
\mathbf{1}\{\mbox{$k$ is prime}\}\log k\ 
=\ \psi(x;q,a)+O(\phi(q)^{-1}x^{1/2}\log x)\,.
\]
So, there is a constant $c>0$ such that 
\[
\frac{1}{2}(\pi(x;q,a)-\pi(x^{1/2};q,a))\log x\  
<\ \psi(x;q,a)+c\phi(q)^{-1}x^{1/2}\log x\,,
\]
and hence, for some constant $c'>0$,
\[
\pi(x;q,a)<\frac{2\psi(x;q,a)}{\log x}+c'\phi(q)^{-1}x^{1/2}\,.
\]
Now, by \cite[equation (14) of Chapter 20]{Dave80}, the Extended
Riemann Hypothesis implies that for $(a,q)=1$,
\[
\psi(x;q,a)=\frac{x}{\phi(q)}+O(x^{1/2}(\log x)^2)\,.
\]
So, for some constant $A>0$, we have
\[
\pi(x;q,a)<\frac{2x}{\phi(q)\log x}+Ax^{1/2}\log x\,.
\]

\end{document}